\documentclass[preprint, a4paper]{elsarticle}

\usepackage{amssymb,amsmath,amsthm, color}
\usepackage[margin=2.8cm]{geometry}
\usepackage{url}
\usepackage{tikz}
\newtheorem{theorem}{Theorem}[section]
\newtheorem{lemma}[theorem]{Lemma}
\newtheorem{definition}[theorem]{Definition}
\newtheorem{corollary}[theorem]{Corollary}

\newtheorem{remark}[theorem]{Remark}
\newtheorem{example}[theorem]{Example}
\newcommand{\F}{\mathbb F}

\newcommand{\N}{\mathbb N}

\newcommand{\rad}{\mathop{\rm rad}\nolimits}
\newcommand{\ord}{\mathop{\rm ord}\nolimits}

\newcommand{\doublespace}

\begin{document}

\begin{frontmatter}

\title{Factorization of composed polynomials and applications}

\author[UFMG]{F.E. Brochero Mart\'inez}
\ead{fbrocher@mat.ufmg.br}
\author[USP]{Lucas Reis\corref{cor1}}
\ead{lucasreismat@gmail.com}
\author[UFMG,UFRR]{Lays Silva}
\ead{laysgrazielle@gmail.com}
\fntext[UFRR]{Permanent address: Departamento de Matem\'{a}tica, Universidade Federal de Roraima, Boa Vista-RR, 69310-000, Brazil}
\address[UFMG]{Departamento de Matem\'{a}tica, Universidade Federal de Minas Gerais, Belo Horizonte, MG, 30123-970, Brazil.}
\address[USP]{Universidade de S\~{a}o Paulo, Instituto de Ci\^{e}ncias Matem\'{a}ticas e de Computa\c{c}\~{a}o, S\~{a}o
Carlos, SP 13560-970, Brazil.}
\cortext[cor1]{Corresponding author}
\begin{abstract}
Let $\mathbb{F}_q$ be the finite field with $q$ elements, where $q$ is a prime power and $n$ be a positive integer. In this paper, we explore the factorization of $f(x^{n})$ over $\mathbb{F}_q$, where $f(x)$ is an irreducible polynomial over $\F_q$. Our main results provide generalizations of recent works on the factorization of binomials $x^n-1$. As an application, we provide an explicit formula for the number of irreducible factors of $f(x^n)$ under some generic conditions on $f$ and $n$.
\end{abstract}

\end{frontmatter}

\section{Introduction}

The factorization of polynomials over finite fields plays an important role in a wide variety of technological situations, including efficient and secure communications, cryptography \cite{cryp},  deterministic simulations of random process, digital tracking system  \cite{Gol} and error-correcting codes \cite{codes}.

For instance, certain codes can be constructed from the divisors of special polynomials; for $\lambda\in \F_q^*$ and $n$ a positive integer,   
each  divisor  of $x^n-\lambda$ in $\F_q[x]$ determines a $\lambda$-constacyclic code of length $n$ over $\F_q$ . In fact,  each $\lambda$-constacyclic code of  length $n$ can be represented as an ideal of the ring ${\mathcal R}_{n,\lambda}=\F_q[x]/\langle x^n-\lambda\rangle$, and each ideal of $\mathcal R_{n,\lambda}$ is generated by a unique factor of $x^n-\lambda$. 
In addition, if $g\in \F_q[x]$ is an irreducible factor of $x^n-\lambda$, the polynomial $\frac{x^n-\lambda}{g(x)}$  generates an  ideal  of $\frac{\F_q[x]}{(x^n-\lambda)}$ that can be see as a  minimal $\lambda$-constacyclic code of lenght $n$ and dimension $\deg(g(x))$.  
In particular, in order to determine every minimal $\lambda$-constacyclic code, the factorization of $x^n-\lambda$ over $\F_q$ is required. Some partial results about this problem can be found in~\cite{LLWZ, LiYu, BaRa}. 
The case $\lambda=1$ corresponds to the well known cyclic codes (see  \cite{LiNi}).  

We observe that the polynomial $x^n-\lambda$ is of the form $f(x^n)$,   where $f(x)=x-\lambda \in \F_q[x]$ is an irreducible polynomial.  
%
%
%
%
The following theorem provides a criterion on the irreducibility of composed polynomials $f(x^n)$ when $f$ is irreducible.

\begin{theorem}[\cite{LiNi} Theorem 3.35]\label{3.35}
Let $n$ be a positive integer and $f\in \mathbb F_{q}[x]$ be an irreducible polynomials of degree $k$ and order $e$. Then the polynomial $f(x^{n})$ is irreducible over $\mathbb F_{q}$ if and only if the following conditions are satisfied:
\begin{itemize}
\item[(1)] $\rad(n)$ divides $e$,
\item[(2)] $\gcd\left(n, \frac {q^{k}-1}{e}\right)=1$ and
\item[(3)] if $4$ divides $n$ then $q^{k}\equiv 1\pmod 4$.
\end{itemize}
In addition, in the case that the polynomial $f(x^{n})$ is irreducible, it has degree $kn$ and order $en$.
\end{theorem}


In the present paper, we  consider $f\in\F_q[x]$  an irreducible polynomial of degree $k$ and order $e$, and impose some condition on $n$ (similar to the ones in \cite{BGO}) that make the polynomial $f(x^n)$ reducible. We then discuss the factorization of $f(x^n)$ over $\F_q$.  In~\cite{BGO}, the authors considered the factorization of $x^n-1\in \mathbb F_q[x]$, where $n$ is a positive integer such that every prime factor of $n$ divides $q-1$. A more general situation is considered by Wu, Yue and Fan~\cite{WYF}.
In the first part of this paper, we generalize these results for polynomial of the form $f(x^n)$, where $f\in \mathbb F_q[x]$ is irreducible. We further consider a more general situation when there exists some prime divisor of $n$ that does not divide $q-1$.

\section{Explicit Factors of $f(x^{n})$}

Throughout this paper, $\mathbb F_q$ denotes the finite field with $q$ elements, where $q$ is a prime power. Also, $\overline{\mathbb F}_q$ denotes the algebraic closure of  $\F_q$. 
For any positive integer $k$, $\rad(k)$ denotes the product of the distinct prime divisors of $k$ and, for each prime divisor $p$ of $k$, $ \nu_{p}(k)$ is the $p$-adic valuation of $k$. In other words, if $k = p_{1}^{\alpha_{1}}p_{2}^{\alpha_{2}}\cdots p_{s}^{\alpha_{s}}$ then $\rad(k) = p_{1}\cdots p_{s}$ and $\nu_{p_{i}}(k) = \alpha_{i}$. 
If $f\in \F_q[x]$ is a polynomial not divisible by $x$, the {\em order} of $f(x)$ is the least positive integer $e$ such that $f(x)$ divides $x^e-1$.
It is known that if $f(x)$ is an irreducible polynomial over $\mathbb F_q[x]$ of degree $m$, then each root $\alpha$ of $f$ is an element of $\mathbb F_{q^{m}}$ and, if $f(x)\ne x$, the order $e$ of $f$ is the multiplicative order of $\alpha$ and $m=\ord_eq$. In
addition, the roots of $f(x)$ are simple and they are given by $\alpha, \alpha^{q}, \cdots,\alpha^{q^{m-1}}$.
Let  $\sigma_{q}$  be denote  the Frobenius automorphism of $\overline{\F}_q$ given by $\sigma_q(\alpha)=\alpha^q$ for any $\alpha\in \overline{\F}_q$.
This automorphism naturally extends to the ring  automorphism
$$\begin{array}{cccc}
 & \! \overline{\mathbb{F}}_{q}[x] & \! \longrightarrow & \! \overline{\mathbb{F}}_{q}[x] \\
& a_mx^m+\cdots+a_1x+a_0 & \! \longmapsto& a_m^qx^m+\cdots+a_1^qx+a_0^q,
\end{array}
$$
that we also denote by $\sigma_q$.

The following well-known result provides some properties on the $p$-adic valuation of numbers of the form $a^j-1$ with $a\equiv 1\pmod p$. Its proof is direct by induction so we omit the details.

\begin{lemma}[Lifting the Exponent  Lemma]\label{LEL} Let $p$ be a prime and $\nu_p$ be the $p$-valuation. The following hold:
\begin{enumerate}[(1)]
\item if  $p$ is an odd prime such that $p$ divides $a-1$ then $\nu_{p}(a^{k}-1)=\nu_{p}(a-1)+\nu_{p}(k)$;
\item if $p=2$  and $a$ is odd number, then
$$
\nu_{2}(a^{k}-1)=
\begin{cases}
\nu_{2}(a-1)&\text{if $k$ is odd,} \\
\nu_{2}(a^{2}-1)+\nu_{2}(k)- 1&\text{if $k$ is even.}\\
\end{cases}
$$
\end{enumerate}
\end{lemma}

In this section, we give an explicit factorization of polynomial $f(x^n)$ over $\mathbb F_q[x]$, where $f(x)$ is an irreducible polynomial in $\mathbb F_q[x]$ and $n$ satisfies some special conditions. This result generalizes the ones found in~\cite{BGO}, where the authors obtain the explicit factorization of $x^{n}-1\in \mathbb F_q[x]$ under the condition that $\rad(n)$ divides $q-1$.
We have the  following preliminary result.

\begin{lemma}\label{lepri1}
Let $n$ be a positive integer such that $\rad(n)$ divides $q-1$ and let $f\in \mathbb{F}_{q}[x]$ be an irreducible polynomial of degree $k$ and order $e$, with $\gcd(ek,n)=1$. Additionally, suppose that $q\equiv 1\pmod 4$ if $8$ divides $n$. Let  $r$ be a positive integer such that $nr\equiv 1\pmod e$.
If $\alpha\in \overline{\F}_q$ is any root of $f(x)$, $\theta\in \mathbb F_q^{*}$ is an element of order $d=\gcd(n,q-1)$ and $t$ is a divisor of  $\frac nd$, the following hold:
\begin{enumerate}[(a)]
\item  the polynomial $g_t(x):=\displaystyle\prod_{i=0}^{k-1} (x-\alpha^{trq^{i}})$ is irreducible over $\mathbb F_q$, has degree $k$ and order $e$;
\item  for each nonnegative  integer $u\ge 0$ such that $\gcd(u, t)=1$, the polynomial  $g_t(\theta^{u}x^{t})$  is irreducible in $\mathbb F_q[x]$ and divides $f(x^{n})$.
\end{enumerate}
\end{lemma}

\begin{proof}
We split the proof into cases.

\begin{enumerate}[(a)]
\item Since $f\in \F_q[x]$ is irreducible of degree $k$ and $\alpha$ is one of its roots, $\alpha\in \mathbb F_{q^k}$ and $\alpha^{q^k}=\alpha$. In particular, $\sigma_{q}(g_{t}(x))=\displaystyle\prod_{i=0}^{k-1} (x-\alpha^{trq^{i+1}})=\displaystyle\prod_{i=0}^{k-1} (x-\alpha^{trq^{i}})$ and so $g_{t}\in \mathbb F_q [x]$. In order to prove that $g_t(x)$ is irreducible, it suffices to show that the least positive integer $i$ for which $\alpha^{rtq^{i}}=\alpha^{rt}$ is $i=k$. We observe that  
$\alpha^{rtq^{i}}=\alpha^{rt}$ if and only if $\alpha^{rt(q^{i}-1)}=1$, that is also equivalent to  
$rt(q^{i}-1)\equiv 0\pmod e$. The result follows from the fact that $\gcd(e,tr)=1$ and $\ord_{e} q=k$.

\item  Suppose that $\lambda$ is a root of $g_{t}(\theta^{u}x^{t})$, or equivalently,  $\lambda^{t}\theta^{u}$ is a root of $g_{t}(x)$. Then there exists $j\in \mathbb N$ such that $\lambda^t\theta^{u} =\alpha^{trq^j}$.  In particular, 
$$\alpha^{q^j}= \alpha^{nrq^j}=\left(\alpha^{trq^j}\right)^{\frac nt}=\left( \lambda^t\theta^{u}\right)^{\frac nt} =\lambda^n.$$ 
Therefore, $\lambda^n=\alpha^{q^j}$ is a root of $f(x)$ and so any root of $g_{t}(\theta^{u}x^{t})$ is also root of $f(x^n)$. This shows that $g_{t}(\theta^{u}x^{t})$ divides $f(x^n)$. Now we are left to prove that $g_t(\theta^{u}x^t)$ is, in fact, irreducible over $\F_q$. Since $g_t(\theta^{u}x)$ is irreducible, we only need to verify that $g_t(\theta^{u}x)$ and $t$ satisfy the conditions of Theorem \ref{3.35}. We observe that each root of $g_{t}(\theta^{u}x)$ is of the form $\theta^{-u}\alpha^{trq^i}$, hence
$$\ord\,g_{t}(\theta^{u}x)=\ord(\theta^{-u}\alpha^{trq^i})=\ord(\alpha^{trq^i}) \cdot\ord(\theta^{-u})=\dfrac{ed}{\gcd (u,\,d)},$$
where in the second equality we use the fact that $\gcd(e,n)=1$ and $d$ divides $n$. Since $\gcd (u, t)=1$ and $\rad(t)$ divides $d$, we have that
$\rad(t)$ divides $\frac{d}{\gcd (u,\,d)}$. In particular, $\rad(t)$ divides $\ord\,g_{t}(\theta^{u}x)$. Next, we verify that $\gcd\left( t,\,\dfrac{q^{k}-1}{\ord\,g_{t}(\theta^{u}x)}\right) =1$.
In fact, $$\gcd\left( t,\,\dfrac{q^{k}-1}{e\dfrac{d}{\gcd (u,d)}}\right) =\gcd \left( t,\,\dfrac{(q^{k}-1)\gcd (u,d)}{ed}\right).$$
In addition,  $\gcd (u,\,t)=1$, $\gcd (n,\,e)=1$ and $\rad(t)$ divides $d$ (that divides $n$). Therefore, $\gcd (t,\,e)=1$ and so
$$\gcd \left( t,\,\dfrac{(q^{k}-1)\gcd (u,d)}{ed}\right)=\gcd \left( t,\,\dfrac{q^{k}-1}{d}\right).$$
Let $p$ be any prime divisor of $t$, hence $p$ also divides $n$. Since $\gcd(n, k)=1$, we have that $p$ does not divide $k$. In particular, $\nu_p(k)=0$ and $k$ is odd if $p=2$. From Lemma~\ref{LEL}, we have that $\nu_{p}(q^{k}-1)=\nu_{p}(q-1)+\nu_{p}(k)$. Therefore, $\nu_{p}(q^{k}-1)=\nu_{p}(q-1)$.  In addition, $t$ divides $\frac {n}{d}$ and so $\nu_p(n)> \nu_p(d)$. Since
$$\nu_{p}(d)=\nu_p(\gcd(n,q-1))=\min\lbrace\nu_{p}(n),\,\nu_{p}(q-1)\rbrace,$$ we conclude that  $\nu_p(d)=\nu_p(q-1)$ and
$\nu_{p}\left(\dfrac{ q^{k}-1}{d}\right) =0$. But $p$ is an arbitrary prime divisor of $t$ and so we conclude that $\gcd \left( t,\,\dfrac{q^{k}-1}{d}\right)=1$. Finally, we need to verify that $q\equiv 1\pmod 4$ if $4$ divides $t$. Since $t$ divides $\frac{n}{\gcd(n, q-1)}$ and every prime divisor of $n$ divides $q-1$, if $4$ divides $t$, it follows that $\gcd(n, q-1)$ is even and $\frac{n}{\gcd(n, q-1)}$ is divisible by $4$. Therefore, $n$ is divisible by $8$. From our hypothesis, $q\equiv 1\pmod 4$.
\end{enumerate}
\end{proof}

The main result of this paper is stated as follows.

\begin{theorem}\label{pri1}
Let $f\in \F_q[x]$ be a monic irreducible polynomial of degree $k$ and order $e$. Let $q, n, r, \theta$  and $g_t(x)$ be as in Lemma \ref{lepri1}. Then $f(x^{n})$ splits into monic irreducible polynomials over $\F_q$ as
 \begin{equation}\label{produtoirreductiveis}
 f(x^{n})=\displaystyle\prod_{t|m}\displaystyle\prod_{{1\leq u\leq d}\atop{\gcd(u,t)=1}}\theta^{-uk} g_t(\theta^{u}x^{t})
 \end{equation}
where $m:=\frac nd=\frac n{\gcd(n,q-1)}$.
\end{theorem}

\begin{proof}
We observe that the factors $g_{t}(\theta^{u}x^{t})$ for $1\leq u\leq d$ with $\gcd (u,\,t)=1$ are all distinct. In fact,  if $g_{t}(\theta^{u}x^{t})=g_{t}(\theta^{w}x^{t})$  for some $1\leq u, w\leq d$ with $\gcd(u, t)=\gcd (w,\,t)=1$, 
there exists a nonnegative integer $i$ such that  $\theta^{-u}\alpha^{tr}=\theta^{-w}\alpha^{trq^{i}}$, that is equivalent to $\theta^{w-u}=\alpha^{tr(q^{i}-1)}$ and then
$$\left( \theta^{u-w}\right) ^{e}=\left( \alpha^{tr(q^{i}-1)}\right)^{e}=1.$$
The last equality entails that the order of $\theta$ divides $(w-u)e$. From hyothesis, $\gcd(n, e)=1$, hence $\gcd(d,e)=1$ and so $d$ divides $w-u$. Since $|w-u|<d$, we necessarily have $w=u$. Let $h(x)$ be the polynomial defined by the product at the RHS of  Eq.~\eqref{produtoirreductiveis}. Since the factors of this product are monic, pairwise distinct. and divide $f(x^n)$, $h(x)$ is monic and divides $f(x^n)$. In particular, the equality $f(x^n)=h(x)$ holds if and only if $h(x)$ and $f(x^n)$ have the same degree, i.e., $\deg(h(x))=nk$. Since $t$ divides $m$, $\rad(t)$ divides $\rad(n)$ (hence divides $q-1$) and so $\rad(t)$ divides $d$. Therefore, for each divisor $t$ of $\frac{n}{d}$, the number of polynomials of the form $\theta^{-uk} g_t(\theta^{u}x^{t})$ with $1\le u\le d$ and $\gcd(u,t)=1$ equals $\dfrac{d}{\rad(t)}\varphi(\rad(t))=\dfrac{d\varphi(t)}{t}$, where $\varphi$ is the {\em Euler Phi function}. In particular, the degree of $h(x)$ equals
$$\sum_{t|m} \dfrac{d\varphi(t)}{t}tk=dk\sum_{t|m}\varphi(t)=dkm=nk.$$
\end{proof}

\begin{corollary}\label{conta1}
Let $f(x)$, $n$ and $m$ be as in Theorem \ref{pri1}.  For each divisor $t$ of $m$, the number of irreducible factors of $f(x^n)$ of degree $kt$ is  $\dfrac{\varphi(t)}{t}.\gcd (n,\,q-1)$. The total number of irreducible factors of $f(x^n)$ in $\mathbb F_q[x]$ is 
$$\gcd (n,\,q-1).\displaystyle\prod_{{p|m}\atop{p\text{ prime}}}\left(1+\nu_{p}(m)\dfrac{p-1}{p}\right) .$$ 
\end{corollary}
\begin{proof}
This result is essentially the same result of Corollary 3.2b in \cite{BGO}. We know, from Theorem~\ref{pri1}, that $f(x^n)$ splits into monic irreducible polynomials as
$$f(x^{n})=\displaystyle\prod_{t|m}\displaystyle\prod_{{1\leqslant u\leqslant d}\atop{\gcd (u,\,t)=1}}\theta^{-ku}g_{t}(\theta^{u}x^{t}).$$
In the proof of Theorem \ref{pri1}, we show that for each divisor $t$ of $m$, the number of polynomials of the form $g_t(\theta^{u}x^{t})$ equals
$$\frac{\varphi(\rad (t))}{\rad (t)}d= \frac{\varphi(\rad (t))}{\rad (t)}\gcd(n, q-1)=\frac{\varphi(t)}{t}\gcd(n, q-1)$$
Therefore, the total number of irreducible factors of $f(x^n)$ equals:
$$\sum_{t|m} \dfrac{\varphi(t)}{t}\gcd (n,\,q-1)=\gcd (n,\,q-1)\sum_{t|m} \dfrac{\varphi(t)}{t}=\gcd (n,\,q-1)\prod_{{p|m}\atop{p\text{ prime}}}\left( 1+\nu_{p}(m)\frac{p-1}{p}\right),$$
where the last equality is obtaned following the same steps as in the proof of Corollary 3.2b in \cite{BGO}. 
\end{proof}

\begin{example}
Let $f(x)=x^{3}+4x^{2}+6x+1\in \mathbb{F}_{11}[x]$ and $n=5^{a}, a\ge 1$. 
We observe that $f(x)$ is an irreducible polynomial of order $14$ over $\mathbb{F}_{11}$. In fact,
$$\Phi_{14}(x)=\dfrac{x^{14}-1}{(x^{7}-1)(x+1)}
=f(x)h(x) \text{ and $\ord_{14}(11) =3$,}$$
where $h(x)=x^{3}+6x^{2}+4x+1$. In particular, $f(x)$ and $n$ satisfy the conditions of Theorem \ref{pri1} and, following the notation of Theorem~\ref{pri1}, we can take $\theta=3$. For $n=5$, we have that
$r=3$ and $$g_{1}(x)=
(x-\alpha)(x-\alpha^{-3})(x-\alpha^{-5})=f(x).$$
Therefore,
$$f(x^5)=\prod_{1\le u \le5}  9^uf(3^ux).$$
If $a\ge 2$, we can take $r=3^{a}$ and for each integer $0\le b\le a-1$,
$$g_{_{5^{b}}}(x)=\begin{cases}
(x-\alpha)(x-\alpha^{9})(x-\alpha^{11})&\text{ if  $a-b$ is odd}\\
(x-\alpha^3)(x-\alpha^{5})(x-\alpha^{13})&\text{ if  $a-b$ is even.}
\end{cases}
$$
Therefore, $$f(x^{5^{a}})=\prod_{1\le u \le5}  9^uf(3^ux)\cdot \displaystyle\prod_{1\le b<a \atop \text{$b$  even}}\displaystyle\prod_{u=1}^4 9^uf(3^{u}x^{5^{b}})\displaystyle\prod_{1\le b<a \atop b\text{ odd}}\displaystyle\prod_{u=1}^4 9^uh(3^{u}x^{5^{b}}).$$
\end{example}

\subsection{Applications of Theorem~\ref{pri1}} Here we present some interesting situations where Theorem~\ref{pri1} gives a more explicit result. We recall that, for a positive integer $e$ with $\gcd(e, q)=1$, the $e$-th cyclotomic polynomial $\Phi_e(x)\in \F_q[x]$ is defined recursively by the identity $x^e-1=\prod_{m|e}\Phi_m(x)$. Of course, any irreducible factor of $\Phi_e(x)$ has order $e$. The following classical result provides the degree distribution of the irreducible factors of cyclotomic polynomials.

\begin{theorem}[\cite{LiNi} Theorem 2.47] \label{thm:lini}
Let $e$ be a positive integer such that $\gcd(e, q)=1$ and set $k=\ord_eq$. Then $\Phi_e(x)$ factors as $\frac{\varphi(e)}{k}$ irreducible polynomials over $\F_q$, each of degree $k$.
\end{theorem}

The following result entails that, if we know the factorization of $\Phi_e(x)$ (resp. $x^e-1$), under suitable conditions on the positive integer $n$, we immediately obtain the factorization of $\Phi_e(x^n)$ (resp. $x^{en}-1$).

\begin{theorem}\label{thm:app}
Let $e$ be a positive integer such that $\gcd(e, q)=1$, set $k=\ord_eq$ and $l=\frac{\varphi(e)}{k}$. Additionally, suppose that $n$ is a positive integer such that $\rad(n)$ divides $q-1$, $\gcd(n, ke)=1$ and $q\equiv 1\pmod 4$ if $n$ is divisible by $8$. If $\Phi_e(x)=\prod_{1\le i\le l}f_i(x)$ is the factorization of $\Phi_e(x)$ over $\F_q$, then the factorization of $\Phi_e(x^n)$ over $\F_q$ is given as follows
\begin{equation}\label{eq:app1}\Phi_e(x^n)=\prod_{i=1}^l\prod_{t|m}\prod_{{1\leq u\leq d}\atop{\gcd(u,t)=1}}\theta^{-uk} f_i(\theta^{u}x^{t}),\end{equation}
where $d=\gcd(n, q-1)$, $m=\frac{n}{d}$ and $\theta\in \F_q$ is any element of order $d$. In particular, if $x^e-1=\prod_{i=1}^{N}F_i(x)$ is the factorization of $x^e-1$ over $\F_q$, then the factorization of $x^{en}-1$ over $\F_q$ is given as follows
\begin{equation}\label{eq:app2}x^{en}-1=\prod_{i=1}^N\prod_{t|m}\prod_{{1\leq u\leq d}\atop{\gcd(u,t)=1}}\theta^{-u\cdot \deg(F_i)} F_i(\theta^{u}x^{t}).\end{equation}
\end{theorem}

\begin{proof}
We observe that, for any $1\le i\le l$, $f_i(x)$ has order $e$ and, from Theorem~\ref{thm:lini}, $f_i$ has degree $k$. In particular, from hypothesis, we are under the conditions of Theorem~\ref{pri1}. For each divisor $t$ of $m$, let $g_{t, i}$ be the polynomial of degree $k$ and order $e$ associated to $f_i$ as in Lemma~\ref{lepri1}. From Theorem~\ref{pri1}, we have that
$$f_i(x^n)=\prod_{t|m}\prod_{{1\leq u\leq d}\atop{\gcd(u,t)=1}}\theta^{-uk} g_{i, t}(\theta^{u}x^{t}),$$
hence
$$\Phi_e(x^n)=\prod_{i=1}^lf_i(x^n)=\prod_{i=1}^l\prod_{t|m}\prod_{{1\leq u\leq d}\atop{\gcd(u,t)=1}}\theta^{-uk} g_{i, t}(\theta^{u}x^{t}).$$

In particular, in order to obtain Eq.~\eqref{eq:app1},it suffices to prove that, for any divisor $t$ of $m$, we have the following equality
\begin{equation}\label{eq:cyclo}\prod_{i=1}^lg_{t, i}(x)=\prod_{i=1}^lf_i(x).\end{equation}
Since each $g_{t, i}$ is irreducible of order $e$, $g_{t, i}$ divides $\Phi_e(x)=\prod_{1\le i\le l}f_i(x)$ and so Eq.~\eqref{eq:cyclo} holds if and only if the polynomials $g_{t, i}$ are pairwise distinct. We then verify that such polynomials are, in fact, pairwise distinct. For this, if $g_{t, i}=g_{t, j}$ for some $1\le i, j\le l$, from definition, there exist roots $\alpha_i, \alpha_j$ of $f_i$ and $f_j$, respectively, such that
$$\alpha_i^{rt}=\alpha_j^{rtq^{h}},$$
for some $h\ge 0$, where $r$ is a positive integer such that $rn\equiv 1\pmod e$. Raising the $\frac{n}{t}$-th power in the last equality, we obtain $\alpha_i=\alpha_j^{q^h}$, i.e., $\alpha_i$ and $\alpha_j^{q^h}$ are conjugates over $\F_q$. Therefore, they have the same minimal polynomial, i.e., $f_i=f_j$ and so $i=j$. This concludes the proof of Eq.~\eqref{eq:app1}. We now proceed to Eq.~\eqref{eq:app2}. From definition,
$$x^{en}-1=\prod_{i=1}^NF_i(x^n)=\prod_{m|e}\Phi_m(x^n).$$
In particular, in order to prove Eq.~\eqref{eq:app2}, we only need to verify that Eq.~\eqref{eq:app1} holds when replacing $e$ by any of its divisors. If $m$ divides $e$, $k'=\ord_mq$ divides $k=\ord_mq$. Therefore, $\gcd(n, ke)=1$ implies that $\gcd(n, k'm)=1$. In particular,  Eq.~\eqref{eq:app1} holds for $e=m$.
\end{proof}

As follows, we provide some particular situations where Theorem~\ref{thm:app} naturally applies. It is well-known that if $P$ is an odd prime number and $a$ is a positive integer such that $\gcd(a, P)=1$ and $\ord_{P^2}a=\varphi(P^2)=P(P-1)$, then $\ord_{P^s}a=\varphi(P^s)$ for any $s\ge 1$ (i.e., if $a$ is a primitive root modulo $P^2$, then is a primitive root modulo $P^s$ for any $s\ge 1$). In particular, from Theorem~\ref{thm:lini}, we have the following well known result.

\begin{corollary}
Let $P$ be an odd prime number such that $\ord_{P^2}q=\varphi(P^2)=P(P-1)$, i.e., $q$ is a primitive root modulo $P^2$. Then, for any $s\ge 1$ and $1\le i\le s$, $\Phi_{P^i}(x)$ is irreducible over $\F_q$. In particular, for any $s\ge 1$, the factorization of $x^{P^s}-1$ over $\F_q$ is given by
$$x^{P^s}-1=(x-1)\prod_{i=1}^s\Phi_{P^i}(x).$$
\end{corollary}

Combining the previous corollary with Eq.~\eqref{eq:app2} we obtain the following result.

\begin{corollary}\label{cor:app}
Let $P$ be an odd prime number such that $\ord_{P^2}q=\varphi(P^2)=P(P-1)$, i.e., $q$ is a primitive root modulo $P^2$. Additionally, let $n$ be a positive integer such that $\rad(n)$ divides $q-1$ and $\gcd(n, P-1)=1$. Set $d=\gcd(n, q-1)$ and $m=\frac{n}{d}$. If $\theta$ is any element in $\F_q$ of order $d$ then, for any $s\ge 1$, the factorization of $\Phi_{P^s}(x^n)$ and  $x^{P^sn}-1$ over $\F_q$ are given as follows
$$\Phi_{P^s}(x^n)=\prod_{t|m}\prod_{{1\leq u\leq d}\atop{\gcd(u,t)=1}}\theta^{u\cdot \varphi(P^s)} \Phi_{P^s}(\theta^{u}x^{t}),$$
and
$$x^{P^sn}-1=\left(\prod_{t|m}\prod_{{1\leq u\leq d}\atop{\gcd(u,t)=1}}(x^t-\theta^u) \right) \times \left(\prod_{1\le i\le s}\prod_{t|m}\prod_{{1\leq u\leq d}\atop{\gcd(u,t)=1}}\theta^{u\cdot \varphi(P^i)} \Phi_{P^i}(\theta^{u}x^{t})\right).$$
\end{corollary}

Corollary~\ref{cor:app} extends Lemma 5.3 in~\cite{MR18}, where the case $n$ a power of a prime $r\ne P$ that divides $q-1$ (but does not divide $P-1$)  is considered. As follows, we provide some examples of Corollary~\ref{cor:app}. 

\begin{example}
Let $q$ be a prime power such that $q-1\equiv 3, 6\pmod 9$ and $q$ is primitive modulo $25=5^2$ (there exist $2\cdot \varphi(20)=16$ residues modulo $225=9\cdot 25$ with this property). We observe that Corollary~\ref{cor:app} applies to $P=5$ and $n=3^a$ with $a\ge 1$. We have that $d=\gcd(q-1, 3^{a})=3$ and $m=3^{a-1}$. If $\theta\in \F_q$ is any element of order $3$, for any $a, s\ge 1$, we have that
$$x^{5^{s}3^{a}}-1=\left(\prod_{t=0}^{a-1}\prod_{{1\leq u\leq 3}\atop{\gcd(u,3^t)=1}}(x^{3^{t}}-\theta^u) \right)\times \left(\prod_{i=1}^{s}\prod_{t=0}^{a-1}\prod_{{1\leq u\leq 3}\atop{\gcd(u, 3^t)=1}}\theta^{-u\cdot \varphi(5^{i})}\Phi_{5^i}(\theta^{u}x^{3^t})\right),$$
where $\Phi_{5^i}(x)=x^{4\cdot 5^{i-1}}+x^{3\cdot 5^{i-1}}+x^{2\cdot 5^{i-1}}+x^{5^{i-1}}+1$ for $i\ge 1$. 
\end{example}

\begin{example}
Let $q$ be a prime power such that $\gcd(q-1, 25)=5$ and $q$ is primitive modulo $9$ (there exist $4\cdot \varphi(6)=8$ residues modulo $225=25\cdot 9$ with this property). We observe that Corollary~\ref{cor:app} applies to $P=3$ and $n=5^a$ with $a\ge 1$. We have that $d=\gcd(q-1, 5^{a})=5$ and $m=5^{a-1}$. If $\theta\in \F_q$ is any element of order $5$, for any $a, s\ge 1$, we have that
$$x^{5^{a}3^{s}}-1=\left(\prod_{t=0}^{a-1}\prod_{{1\leq u\leq 5}\atop{\gcd(u,5^t)=1}}(x^{5^{t}}-\theta^u) \right)\times \left(\prod_{i=1}^{s}\prod_{t=0}^{a-1}\prod_{{1\leq u\leq 5}\atop{\gcd(u, 5^t)=1}}\theta^{-u\cdot \varphi(3^{i})}\Phi_{3^i}(\theta^{u}x^{5^t})\right),$$
where $\Phi_{3^i}(x)=x^{2\cdot 3^{i-1}}+x^{3^{i-1}}+1$ for $i\ge 1$.
\end{example}

\begin{example}
Let $q$ be a prime power such that $\gcd(q-1, 225)=15$ and $q$ is primitive modulo $289=17^2$ (there exist $\varphi(15)\cdot \varphi(17\cdot 16)=1024$ residues modulo $225\cdot 289$ with this property). We observe that Corollary~\ref{cor:app} applies to $P=17$ and $n=3^a5^{b}$. For simplicity, we only consider the case $a, b\ge 1$ (the cases $a=0$ or $b=0$ are similarly treated). We have that $d=\gcd(q-1, 3^a5^b)=15$ and $m=3^{a-1}5^{b-1}$. If $\theta\in \F_q$ is any element of order $15$, for any $a, b, s\ge 1$, we have that

\begin{align*}x^{3^a5^b17^s}-1=&\left(\prod_{{0\le t_1\le a-1}\atop{0\le t_2\le b-1}}\prod_{{1\leq u\leq 15}\atop{\gcd(u, 3^{t_1}5^{t_2})=1}}(x^{3^{t_1}5^{t_2}}-\theta^u) \right)\times \\ {}&\left(\prod_{1\le i\le s}\prod_{{0\le t_1\le a-1}\atop{0\le t_2\le b-1}}\prod_{{1\leq u\leq 15}\atop{\gcd(u, 3^{t_1}5^{t_2})=1}}\theta^{-u\cdot \varphi(17^{i})}\Phi_{17^i}(\theta^{u}x^{3^{t_1}5^{t_2}})\right),\end{align*}
where $\Phi_{17^i}(x)=\sum_{j=0}^{16}x^{j\cdot 17^{i-1}}$ for $i\ge 1$.

\end{example}

\section{The general case}
In the previous section, we provided the factorization of $f(x^n)$ over $\F_q$ under special conditions on $f$ and $n$. In particular, we assumed that $\rad(n)$ divides $q-1$ and $q\equiv 1\pmod 4$ if $n$ is divisible by $8$. In this section, we show a natural theoretical procedure to extend such result, removing these conditions on $n$. In order to do that, the following definition is useful.
\begin{definition}
Let $n$ be a positive integer such that $\gcd(n,q)=1$ and set $S_n=\ord_{\rad(n)}q$. Let $s_n$ be the positive integer defined as follows
$$s_n:=\begin{cases} 
S_n & \text{if $q^{S_n}\not\equiv 3 \pmod 4$ or $n\not\equiv 0\pmod 8$,}\\
 2S_n&\text{otherwise.}
 \end{cases}
$$
\end{definition}

In particular, the previous section dealt with positive integers $n$ for which $s_n=1$. Throughout this section, we fix $f\in \F_q[x]$ an irreducible polynomial of degree $k$ and order $e$ and consider positive integers $n$ such that $\gcd(n, ek)=1$, $s_n>1$ and $\gcd(s_n, k)=1$. In addition, as a natural extension of the previous section, $d$ denotes the number $\gcd(n, q^{s_n}-1)$ and $m=\frac nd= \frac n{\gcd( n, q^{s_n}-1)}$. 

We observe that, since $\gcd(s_n, k)=1$, the polynomial $f$ is also irreducible over $\mathbb F_{q^{s_n}}$. Let $\alpha$ be any root of $f$. From Theorem~\ref{pri1},  the irreducible factors of $f(x^n)$ in $\mathbb F_{q^{s_n}}[x]$ are the polynomials 
\begin{equation}\label{G_tu}
G_{t,u}(x):=\displaystyle\prod_{i=0}^{k-1} (x^t-\theta^{-u}\alpha^{trq^{is_n}}), 
\end{equation}
where 
\begin{itemize}
\item $r$ is a positive integer such that $rn\equiv 1\pmod e$;
\item $t$ is a divisor of $m$; 
\item $\theta\in \mathbb F_{q^{s_n}}^*$ is any element of order  $d$;
\item $\gcd(t,u)=1$, $1\le u\le d$. 
\end{itemize}

Now, for each polynomial $G_{t,u}$, we need to determine what is the smallest extension of $\F_q$ that contains its coefficients. This will provide the irreducible factor of $f(x^n)$ (over $\F_q$) associated to $G_{t, u}$.

\begin{definition}
For $t$ and $u$ as above, let $l_{t,u}$ be the least positive integer $v$ such that $G_{t,u}(x) \in  \mathbb F_{q^v}[x]$. 
\end{definition}

\begin{remark}
Since $G_{t,u}(x) \in  \mathbb F_{q^{s_n}}[x]$, we have that $l_{t,u}$ is a divisor of $s_n$.  Using the Frobenius automorphism, we conclude that  every irreducible factor of $f(x^n)$ in $\mathbb F_q[x]$ is of the form
$$\prod_{j=0}^{l_{t,u}-1} \sigma_q^j(G_{t,u}(x)).$$
\end{remark}

\begin{remark}\label{fatoresirredutiveis}
We observe that
$$\prod_{j=0}^{l_{t,u}-1} \sigma_q^j(G_{t,u}(x))=\prod_{j=0}^{l_{t,u}-1}\prod_{i=0}^{k-1} (x^t-\sigma_q^j(\theta^{-u}\alpha^{trq^{is_n}})),$$
is a polynomial of weight (i.e., the number of nonzero coefficients is) at most $$k\cdot l_{t, u}+1\le k\cdot s_n+1.$$ In particular, if $f(x)=x-1$, the weight of every irreducible factor of $x^n-1$ is  at most $s_n+1$; the cases $s_n=1$ and $s_n= 2$ are treated in~\cite{BGO}, where the irreducible factors are binomials and trinomials, respectively.
\end{remark}

The following lemma provides a way of obtaining the numbers $l_{t, u}$. In particular, we observe that they do not depend on $t$.

\begin{lemma}\label{l_tu}
The number $l_{t,u}$ is the least positive integer $v$ such that $\dfrac{\gcd(n, q^{s_n} -1)}{\gcd(n, q^{v}-1)}$ divides $u$.
\end{lemma}

\begin{proof}
By definition,  $l_{t,u}$  is the least positive integer $v$ such that $G_{t,u}\in \mathbb{F}_{q^{v}}[x]$.  This last condition is equivalent to 
$$(\theta^{-u}\alpha^{-tr})^{q^{v}}=\theta^{-u}\alpha^{-trq^{is_{n}}}\,\,\,\, \text{for some integer  $i$}.$$
Therefore, $\theta^{-u(q^{v}-1)}=\alpha^{-tr(q^{is_{n}}-q^{v})}$. In particular, we have that $\ord (\theta^{-u(q^{v}-1)})=\ord(\alpha^{-tr(q^{is_{n}}-q^{v})})$. 
Since the orders of $\theta$ and $\alpha$ are relative prime, we conclude that 
$$\frac{d}{\gcd(d,u(q^{v}-1))}= \frac{e}{\gcd(e,tr(q^{is_{n}}-q^{v}))}=1.$$
In particular, $d=\gcd(n,q^{s_{n}}-1)$ divides $u(q^{v}-1)$. Since $v$ divides $s_n$, $q^{v}-1$ divides $q^{s_n}-1$ and so we have that $\gcd(\gcd(n,q^{s_{n}}-1), q^v-1)=\gcd(n,q^{v}-1)$. Therefore, $\frac{\gcd(n,q^{s_{n}}-1)}{\gcd(n,q^{v}-1)}$ divides $u$. Conversely, if $v$ is any positive integer such that $\frac{\gcd(n,q^{s_{n}}-1)}{\gcd(n,q^{v}-1)}$ divides $u$, we have that $d=\gcd(n, q^{s_n}-1)$ divides $u(q^v-1)$ and so $\theta^{-u(q^{v}-1)}=1$. In particular, for any positive integer $i$ such that $is_n\equiv v\pmod k$ (since $\gcd(k, s_n)=1$, such an integer exists), we have that 
$$\alpha^{q^{is_n}-q^v}=1=\theta^{-u(q^v-1)}.$$
From the previous observations, we conclude that $G_{t, u}\in \F_{q^v}[x]$.
\end{proof}

\begin{definition} For each divisor $s$ of $s_n$, set
$$\Lambda_t(s)=|\{G_{t,u}\in \F_{q^s}[x]|  \text{$G_{t,u}$ divides $f(x^n)$}\}|$$
and
$$\Omega_t(s)=|\{G_{t,u}\in \F_{q^s}[x]|  \text{$G_{t,u}$  divides $f(x^n)$ and $G_{t,u}\notin \F_{q^v}[x]$ for any $v<s$}\}|,$$
where the polynomials $G_{t,u}$ are given by the formula \eqref{G_tu}.
\end{definition}

We obtain the following result.

\begin{lemma} Let $t$ be a positive divisor of $m$ and $r_{n,t}$ be the smallest positive divisor of $s_n$ such that $\gcd \left(\dfrac{\gcd( n, q^{s_n}-1)}{\gcd( n, q^{r_{n,t}}-1)},t\right)= 1$.  For an arbitrary divisor $s$ of $s_n$, the following hold:
\begin{enumerate}[(a)]
\item if 
$r_{n,t}$ does not divide $s$, we have that $\Lambda_t(s)=\Omega_t(s)=0$;
\item If  
$r_{n,t}$ divides $s$, we have that 
$$\Lambda_t(s)=\frac {\varphi(t)}t\gcd (n, q^s-1)$$
and
$$\Omega_t(s)=\frac {\varphi(t)}t\sum_{r_{n,t}|v|s}\mu\left(\frac sv\right)\gcd (n, q^v-1),$$
where $\mu$ is the M\"obius function.
\end{enumerate}
\end{lemma}

\begin{proof}
From Lemma~\ref{l_tu}, we have that $G_{u,t}\in \mathbb{F}_{q^{s}}[x]$ if and only if  $\frac{\gcd(n, q^{s_{n}}-1)}{\gcd(n,q^{s}-1)}$ divides $u$. Since $\gcd(u,t)=1$, if $G_{u,t}(x)$ is in $\mathbb{F}_{q^{s}}[x]$, it follows that $\gcd\left(\frac{\gcd(n, q^{s_{n}}-1)}{\gcd(n,q^{s}-1)}, t\right)=1$. Suppose, by contradiction,  that $r_{n,t}$ does not divide $s$ and $\Lambda_t(s)\ne 0$. Therefore, we necessarily have that
$$\gcd\left(\frac{\gcd(n, q^{s_{n}}-1)}{\gcd(n,q^{s}-1)}, t\right)=\gcd\left(\frac{\gcd(n, q^{s_{n}}-1)}{\gcd(n,q^{r_{n,t}}-1)}, t\right)=1$$
However, for each prime divisor $p$ of $t$, we have that
$$\min\{\nu_p(n),\nu_p(q^{s_n}-1)\}=\min\{\nu_p(n),\nu_p(q^{r_{n,t}}-1)\}=\min\{\nu_p(n),\nu_p(q^{s}-1)\}\ge 1.$$
If $s'=\gcd(r_{n,t},s)$, then there exist positive integers $a$ and $b$ such that $s'=ar_{n,t}-bs$ and
\begin{align*}
\nu_p(q^{s'}-1)&=\nu_p(q^{ ar_{n,t}-bs}-1)=\nu_p(q^{ ar_{n,t}}-q^{bs})\\
&\ge\min\{\nu_p(q^{ ar_{n,t}}-1), \nu_p(q^{bs}-1)\}\\
&\ge \min\{\nu_p(q^{r_{n,t}}-1), \nu_p(q^{s}-1)\}.
\end{align*}
In particular, $\min\{\nu_p(n),\nu_p(q^{s_n}-1)\}=\min\{\nu_p(n),\nu_p(q^{s'}-1)\}$ for every prime divisor of $t$, a contradiction since $s'<r_{n,t}$.  Now, let $s$ be a positive divisor of $s_n$ such that $r_{n,t}|s$.  Every  factor  $G_{u,t}$ of $f(x^n)$  that is in $\mathbb{F}_{q^{s}}$ satisfies the conditions $\gcd(u,t)=1$, $1\le u\le \gcd(n, q^{s_n}-1)$  and $\frac{\gcd(n, q^{s_{n}}-1)}{\gcd(n,q^{s}-1)}$ divides $u$. Then 
$u= \frac{\gcd(n, q^{s_{n}}-1)}{\gcd(n,q^{s}-1)}u'$, with  $\gcd(u',t)=1$, $1\le u'\le \gcd(n, q^{s}-1)$. 

In addition, if $p$ is a prime divisor of $t$, we have that $p$ divides $m$ and then $\nu_p(n)>\nu_p(q^{s_n}-1)\ge 1$. However, $p\nmid \frac{\gcd(n,q^{s_n}-1)}{\gcd(n,q^{s}-1)}$ and so $p$ divides $q^s-1$.  It follows that $\rad(t)$ divides $q^s-1$. We conclude that the number of elements $u'$ satisfying the previous conditions equals
$$\Lambda_{t}(s)=\varphi(\rad(t))\cdot\frac{\gcd(n, q^{s}-1)}{\rad(t)} = \frac{\varphi(t)}t  \gcd(n, q^{s}-1).$$
Finally, we observe that $\Lambda_{t}(s)=\sum_{v|s}\Omega_{t}(v)$  and so the M\"obius inversion formula yields the following equality
$$\Omega_{t}(s)= \sum_{v|s}\mu\left(\frac{s}{v}\right)\Lambda_{t}(v)=\sum_{r_{n,t}|v|s}\mu\left(\frac{s}{v}\right)\frac{\varphi{(t)}}{t}\gcd(n, q^{v}-1).$$ 
\end{proof}

\begin{theorem} Let $f\in \F_q[x]$ be a irreducible polynomial of degree $k$ and order $e$. Let $n$ be a positive integer such that $\gcd(n,ek)=1$ and $\gcd(k, s_n)=1$. 
The number of irreducible factors of $f(x^n)$ in $\F_q[x]$ is
$$\frac 1{s_n} \sum_{t|m}  \frac {\varphi(t)}{t} \sum_{r_{n,t}| v|s_n} \gcd(n, q^v-1) \varphi\left(\frac {s_n}{v}\right),$$
 or equivalently
$$\frac 1{s_n} \sum_{v|s_n} \gcd(n, q^v-1) \varphi\left(\frac {s_n}{v}\right) \prod_{t|m_v} \left(1+\nu_p(m_v) \frac {p-1}p\right).$$
where $m_v=\max\left\{t\left| \ \text{$t$ divides $m$  and $\gcd\left(\frac{\gcd(n, q^{s_n}-1)}{\gcd(n, q^v-1)}, t\right)=1$}\right.\right\}$.


\end{theorem}

\begin{proof}
The number of irreducible factors of $f(x^n)$ in $\F_q[x]$ is
\begin{align*}
\sum_{t|m}\frac {\varphi(t)}{t} \sum_{s|s_n} \frac 1s \Omega_t(s)&=\sum_{t|m}\frac{\varphi(t)}{t} \sum_{r_{n,t}|s|s_{n}}\frac{1}{s}\sum_ {v'|\frac{s}{r_{n,t}}}\mu\left( \frac{s}{r_{n,t}v'}\right) \gcd(n, q^{v'r_{n,t}}-1)\\ &= \sum_{t|m}\frac{\varphi(t)}{t} \sum_{s'|\frac{s_{n}}{r_{n,t}}}\frac{1}{r_{n,t}s'}\sum_ {v'|s'}\mu\left( \frac{s'}{v'}\right) \gcd(n, q^{v'r_{n,t}}-1)\\ &= \sum_{t|m}\frac{\varphi(t)}{t} \sum_{v'|\frac{s_{n}}{r_{n,t}}}\dfrac{\gcd(n, q^{v'r_{n,t}}-1)}{r_{n,t}}\sum_ {v'|s'|\frac{s_{n}}{r_{n,t}}}\dfrac{\mu\left(\frac{s'}{v'}\right)}{{s'}}\\ &=\sum_{t|m}\frac{\varphi(t)}{t} \sum_{v'|\frac{s_{n}}{r_{n,t}}}\dfrac{\gcd(n, q^{v'r_{n,t}}-1)}{v'r_{n,t}}\sum_ {s''|\frac{s_{n}}{v'r_{n,t}}}\frac{\mu(s'')}{{s''}}\\ 
&= \sum_{t|m}\frac{\varphi(t)}{t}  \sum_{r_{n,t}|v|s_{n}}\dfrac{\gcd(n, q^{v}-1)}{v}\sum_ {s|\frac{s_{n}}{v}}\frac{\mu(s)}{s}\\
&= \frac 1{s_n}\sum_{t|m}\frac{\varphi(t)}{t}  \sum_{r_{n,t}|v|s_{n}}\gcd(n, q^{v}-1)\varphi\left(\frac{s_{n}}{v}\right).
\end{align*}
\end{proof}

\section{Factors of $f(x^n)$ when  $s_n$ is a prime number}
 In this section, as an application of previous results,  we consider some cases when $s_n$ is a prime number.   Throughout this section, for each $\gcd(q,d)=1$,   $\sim_{d}$ denotes  the equivalent relation defined as follows: $a\sim_{d} b$ if there exists $j\in \N$ such that $a\equiv bq^j\pmod d$.

\subsection{The case $\rad(n)|(q-1)$ and $s_n\ne 1$ }
This case is the complementary of Theorem~\ref{pri1}, where $q\equiv 3\pmod 4$ and $8|n$ but keeping the condition $\rad(n)|(q-1)$. 

\begin{theorem}\label{pri2}
Let $n$ be an integer and $q$ be power of a prime such that $8|n$ and $q\equiv 3 \pmod 4$. Let  $\theta$ be an element  in $\mathbb F_{q^{2}}^{\ast}$ with order  $d=\gcd (n,q^{2}-1)$ and $f(x)$ be an irreducible polynomial of degree $k$ odd and order $e$. In addition, let  
$g_t(x)$ be as in Lemma \ref{lepri1} .
 Then  $d=2^l\gcd(n,q-1)$, where $l=\min\lbrace\nu_{2}(\frac{n}{2}),\,\nu_{2}(q+1)\rbrace$ and 
 $f(x^{n})$ splits into irreducible factors  in $\mathbb F_q[x]$ as
\begin{equation}\label{fatores2}
f(x^{n})=\displaystyle\prod_{{t|m}\atop{t\text{ odd}}}\displaystyle\prod_{{1\leqslant w\leqslant \gcd(n,q-1)}\atop{\gcd (w,\,t)=1}}\beta^{-wk}g_{t}(\beta^{w}x^{t})\displaystyle\prod_{t|m}\displaystyle\prod_{u\in R_{t}}[\theta^{-uk(q+1)}g_{t}(\theta^{u}x^{t})g_{t}(\theta^{uq}x^{t})],
\end{equation}
where $m=\frac nd =\frac n{\gcd(n,q^2-1)}$, $\beta=\theta^{2^l}$ is  a $\gcd(n,q-1)$-th primitive root of the unity and
$\mathcal R_{t}$ is the set of $q$-cyclotomic classes
$$\mathcal R_{t}=\left\lbrace u\in\mathbb N| 1\le u\le d, \, \gcd(u,t)=1,\,2^l\nmid u
\right\rbrace /\sim_{d},
$$
In addition, the total number of irreducible factors of $f(x^{n})$ in $\mathbb F_q[x]$ is:
$$\gcd (n,\,q-1)\left(\frac 12+2^{l-2}(2+\nu_2(m))\right)\prod_{{p|m}\atop{p\text{ prime}}}\left(1+\nu_{p}(m)\dfrac{p-1}{p}\right).$$ 
\end{theorem}

\begin{proof}
Since $q\equiv 3\pmod 4$ and $8|n$, we have that $s_n=2$, 

$$d=\gcd (n,\,q^{2}-1)=2\gcd \left( \dfrac{n}{2},\,\dfrac{q-1}{2}(q+1)\right) =2\gcd \left( \dfrac{n}{2},\,\dfrac{q-1}{2}\right)\gcd \left( \dfrac{n}{2},\,q+1\right),$$
where 
$$\gcd \left( \dfrac{n}{2},\,q+1\right)=2^{\min\lbrace\nu_{2}(\frac{n}{2}),\,\nu_{2}(q+1)\rbrace}=2^l.$$

It follows from Lemma \ref{l_tu} that $G_{t u}(x)=\theta^{-uk}g_{t}(\theta^{u}x^{t})\notin \F_q[x]$ if $\dfrac{\gcd(n, q^2-1)}{\gcd(n,q-1)}=2^l$ does not divide $u$, i.e.  the class of $u$  is in $\mathcal R_t$. Then the factorization in  Eq.~\eqref{fatores2} follows from Remark  \ref{fatoresirredutiveis}. Since $r_{n,t}$ is a divisor of $s_n=2$, in order to determine how many factors of each degree we need to consider two cases: $r_{n,t}=1$ or $r_{n, t}=2$. If $r_{n, t}=1$, we have that $t$ satifies the following equalities 
$$\gcd\left(\frac {\gcd(n, q^2-1)}{\gcd(n,q-1)}, t\right)=\gcd(2^l, t)=1,$$
and so $t$ must be odd. If $r_{n,t}=2$, since $r_{n,t}$ is minimal with this property, it follows that $t$ must be even.  In conclusion, the number of irreducible factor of $f(x^n)$ in $\F_q[x]$ equals

\begin{align*}
&\frac 12 \sum_{v|2} \gcd(n, q^v-1) \varphi\left(\frac {2}{v}\right) \prod_{t|m_v} \left(1+\nu_p(m_v) \frac {p-1}p\right)\\
&=\frac  12 \gcd(n, q-1)\prod_{p|m_1}  \left(1+\nu_p(m) \frac {p-1}p\right)+ \frac  12  \gcd(n, q^2-1)\prod_{p|m_2}  \left(1+\nu_p(m) \frac {p-1}p\right)\\
&=\frac  12 \gcd(n, q-1)\prod_{p|m\atop p\ne 2}  \left(1+\nu_p(m) \frac {p-1}p\right)+ \frac  12  \gcd(n, q^2-1)\prod_{p|m}  \left(1+\nu_p(m) \frac {p-1}p\right)\\
&=\frac  12 \gcd(n, q-1) \left( 1+ 2^l\left(1+\frac 12\nu_2(m)\right)\right)\prod_{p|m\atop p\ne 2}  \left(1+\nu_p(m) \frac {p-1}p\right) .
\end{align*}
\end{proof}

\subsection{The case $\rad(n)| (q^p-1)$ with $p$ an odd prime 
and $\rad(n)\nmid (q-1)$}

Since the case $\rad(n)|(q-1)$ is now completely described, we only consider the case that $n$ has at least one prime factor that does not divide $q-1$.

\begin{theorem}\label{teo5}  Let $n$ be a positive integer such that $\ord_{\rad(n)}q=p$ is an odd prime and let $f\in \mathbb F_{q}[x]$ be an irreducible polynomial of degree $k$ and order $e$ such that  $\gcd (ke, n)=\gcd(k, p)=1$. In addition, suppose that $q\equiv 1\pmod4$ if $8|n$. 
Let $\theta$ be an element in $\mathbb{F}_{q^{p}}^{\ast}$ with order $d=\gcd (n,q^{p}-1)$, $m=\frac nd=\frac{n}{\gcd (n,q^{p}-1)}$ 
and $g_{t}(x)$ be as in Lemma \ref{lepri1}. The polynomial $f(x^{n})$ splits into irreducible factors in $\mathbb{F}_{q}[x]$ as
$$\prod_{{t|m}\atop{\rad(t)|(q-1)}}\prod_{{1\leqslant v \leqslant d'}\atop{\gcd (v,\,t)=1}}\beta^{-vk}\,g_{t}(\beta^{v}x^{t})
\prod_{{t|m}\atop{\rad(t)\nmid(q-1)}}\prod_{u\in R_{t}}\left[\theta^{-uk(1+\cdots +q^{p-1})}g_{t}(\theta^{u}x^{t})\cdots 
g_{t}(\theta^{uq^{p-1}}x^{t})\right],$$
where 
\begin{enumerate}[1)]
\item $d'=\gcd(n,\,q-1)$,  $\beta=\theta^{\gcd(n, \frac{q^p-1}{q-1})}$ is a $d'$ primitive root of the unity  and
 $$R_{t}=\left\lbrace u\in \mathbb{N}\left|  1\leq u\leq d, \quad \gcd\left(n,\frac{q^{p}-1}{q-1}\right)\nmid u, \quad \gcd (u,\,t)=1\right.\right\rbrace/\sim_d$$ 
 in the case that  $p\nmid n$ or $p\nmid (q-1)$ or $\nu_{p}(n)>\nu_{p}(q-1)\geq 1$.

\item $d'=p\cdot \gcd(\frac np, q-1)$, $\beta=\theta^{\gcd\left( \frac{n}{p},\,\frac{1}{p}\frac{q^{p}-1}{q-1}\right)}$ is a $d'$ primitive root of the unity  and
$$R_{t}=\left\lbrace  u\in \mathbb{N}\left| 1\leq u\leq d,\quad \gcd\left( \frac{n}{p},\,\frac{1}{p}\frac{q^{p}-1}{q-1}\right)\nmid u,\quad \gcd(u,t)=1\right. \right\rbrace/\sim_d$$
in the case that $p$ divides $n$ and $(q-1)$, and $\nu_{p}(n)\leq \nu_{p}(q-1)$.
\end{enumerate}
\end{theorem}

\begin{proof}
Since $\gcd(k,p)=1$, $f(x)$ is also irreducible in $\mathbb F_{q^p}[x]$. In addition,
$q^{p}\equiv 1\pmod 4$ if $8|n$ and so Theorem \ref{pri1} entails that   $f(x^{n})$ splits into irreducible factors in $\mathbb{F}_{q^{p}}[x]$ as
$$f(x^{n})=\displaystyle\prod_{t|m}\displaystyle\prod_{{1\leqslant u\leqslant d}\atop{\gcd(u,\,t)=1}}\theta^{-uk}\,g_{t}(\theta^{u}x^{t}),$$
where $\theta$ is an element in $\mathbb{F}_{q^{p}}^{\ast}$ with order $d=\gcd (n,q^{p}-1)$.
Following the same inicial step of the proof of  Theorem \ref{pri2}, we observe that  $\theta^{-uk}\,g_{t}(\theta^{u}x^{t})\in \mathbb F_{q}[x]$ if and only if $d|u(q-1)$. At this point, we have two cases to consider:
\begin{enumerate}[1.]
\item $p\nmid n$ or $p\nmid (q-1)$:  
in this case, we have that
$$\gcd(n,\,q^{p}-1)=\gcd\left( n,\,\frac{q^{p}-1}{q-1} (q-1)\right) =\gcd\left( n,\,\frac{q^{p}-1}{q-1}\right). \gcd(n,q-1).$$ Hence
$\gcd(n,\,q^{p}-1)|u(q-1)$ if and only if $\gcd\left( n,\,\frac{q^{p}-1}{q-1}\right)|u$.

\item $p|n$ and $p|(q-1)$: in this case, we have that

\begin{align*}\gcd(n,\,q^{p}-1)= &\gcd\left( n,\frac{q^{p}-1}{q-1} (q-1)\right)=p\cdot \gcd\left( \frac{n}{p},\frac{1}{p}\left( \frac{q^{p}-1}{q-1}\right) (q-1)\right)\\=&
p\cdot \gcd\left( \frac{n}{p},\frac{1}{p}\frac{q^{p}-1}{q-1}\right)\cdot \gcd\Bigl(\frac{n}{p},q-1\Bigr).\end{align*}

Therefore
$\gcd (n,\,q^{p}-1)|u(q-1)$ if and only if 
\begin{equation}\label{lays}
p \cdot\gcd \left( \frac{n}{p},\,\frac{1}{p}\frac{q^{p}-1}{q-1}\right)\cdot \gcd \left(\frac{n}{p},q-1\right)|{u(q-1)}.\end{equation}
We split into subcases:
\begin{enumerate}[2.1.]
\item If $\nu_{p}(n)\leq \nu_{p}(q-1)$, Eq.~\eqref{lays} is equivalent to 
$\gcd \left( \frac{n}{p},\frac{1}{p}\frac{q^{p}-1}{q-1}\right)|u$.
\item If $\nu_{p}(n)> \nu_{p}(q-1)$, Eq.~\eqref{lays} is equivalent to $\gcd \left(n,\frac{q^{p}-1}{q-1}\right)|u.$
\end{enumerate}
\end{enumerate}

We observe that in the cases 1 and 2.2,  the conclusion is the same. Therefore, in these cases,
$\theta^{-uk}g_{t}(\theta^{u}x^{t})\in \mathbb{F}_{q}[x]$ if and only if 
 $u=\gcd \left( n,\,\frac{q^{p}-1}{q-1}\right)\cdot v$ for some positive integer $v$. 
Since $1\leq u \leq \gcd (n,q^{p}-1)$, it follows that  
$1\leq v \leq \gcd (n,q-1)$.
In addition, since $\gcd (u,\,t)=1$, we have that $\gcd(v,t)=1$ and $\nu_{p'}(q^{p}-1)=\nu_{p'}(q-1)$ for every prime $p'|t$. Thus, in the cases 1 and 2.2, the numbers $v$ and $t$ satisfy  the conditions  $\gcd(v,t)=1$ and  $\rad(t)|(q-1)$. 

Now, in the case 2.1, $\theta^{-uk}g_{t}(\theta^{u}x^{t})\in \mathbb{F}_{q}[x]$
if and only if 
$u=\gcd\left( \frac{n}{p},\,\frac{1}{p}\frac{q^{p}-1}{q-1}\right)\cdot v$ for some integer $v$. Since $1\leq u \leq \gcd(n,q^{p}-1)$, then 
$1\leq v\leq p\gcd\left( \frac{n}{p},q-1\right) $. In the same way as before,  we obtain that $v$ and $t$ satisfy the  conditions  $\gcd(v,t)=1$ and  $\rad(t)|(q-1)$. 

In conclusion, we have that the irreducible factors of $f(x^n)$ in $\mathbb F_{q^p}[x]$, that are also irreducible in $\mathbb F_q[x]$  are the form
$$\theta^{-uk}g_{t}(\theta^{u}x^{t})= \beta^{-vk}g_t(\beta^v x^t),$$
where $\rad(t)|(q-1)$ and $\beta$ is an element of $\mathbb F_q^*$ of order $d'$.

Finally, if $\rad(t)\nmid (q-1)$, we have that $\theta^{-uk}g_{t}(\theta^{u}x^{t})\in\mathbb F_{q^p}[x]\setminus \mathbb{F}_{q}[x]$. 
In particular, since the polynomial $$\prod_{i=0}^{p-1}\theta^{-ukq^i}g_{t}(\theta^{uq^i}x^{t}),$$
 is invariant by the Frobenius automorphism and is divisible by $\theta^{-uk}g_{t}(\theta^{u}x^{t})$, such polynomial is monic irreducible over  $\mathbb F_q$.  
\end{proof}

The following result provide information on the number of irreducible factors of $f(x^n)$ under the conditions of Theorem~\ref{teo5}.

\begin{theorem} \label{teo6} Let $n, m, q$ and $f\in \F_q[x]$ be as in  Theorem \ref{teo5}. Then the  number of irreducible factors of $f(x^{n})$ in $\mathbb F_q[x]$ is

$$\frac{p-1}{p}\gcd(n,\,q-1)\displaystyle\prod_{{p'|m}\atop{p'\text {prime}}}\left(1+\nu_{p'}(m)\dfrac{p'-1}{p'}\right)+ \frac{\gcd(n,\,q^{p}-1)}{p}\displaystyle\prod_{{p'|m}\atop{p'\text{ prime}}}\left(1+\nu_{p'}(m)\dfrac{p'-1}{p'}\right).$$ 
\end{theorem}
\begin{proof}
From  Theorem \ref{conta1}, the number of irreducible factors (over $\mathbb F_{q^{p}}[x]$) of degree $kt$ for each $t|m$  is 
$\dfrac{\varphi(t)}{t}\,\gcd(n,\,q^{p}-1).$
 Therefore the total number of irreducible factors is:
\begin{align*}
&\sum_{t|m}\frac{\varphi(t)}{t}\gcd(n,\,q-1)+\sum_{t|m}\frac{1}{p}\left( \frac{\varphi(t)}{t}\gcd(n,\,q^{p}-1)-\frac{\varphi(t)}{t}\gcd(n,\,q-1)\right)\\
&=\sum_{t|m}\frac{p-1}{p}\gcd(n,\,q-1)\frac{\varphi(t)}{t}+\sum_{t|m}\frac{1}{p} \gcd(n,\,q^{p}-1)\frac{\varphi(t)}{t}\\
&= \frac{p-1}{p}\gcd(n,\,q-1)\displaystyle\prod_{{p'|m}\atop{p'\text{prime}}}\left(1+\nu_{p'}(m)\dfrac{p'-1}{p'}\right)+ \frac{\gcd(n,\,q^{p}-1)}{p}\displaystyle\prod_{{p'|m}\atop{p'\text{prime}}}\left(1+\nu_{p'}(m)\dfrac{p'-1}{p}\right)
\end{align*}
\end{proof}

\begin{center}{\bf Acknowledgments}\end{center}
The first author was partially supported by FAPEMIG (Grant number: APQ-02973-17). 
The second author was supported by FAPESP 2018/03038-2, Brazil.


\end{document}